\documentclass{amsart}
\usepackage{amssymb,mathrsfs,MnSymbol}%,skak}
\usepackage{color}

\def\bB {\mathbf{B}}

\def\bE {\mathbf{E}}

\def\bP {\mathbf{P}}

\def\bR {\mathbf{R}}
\def\bS {\mathbf{S}}

\def\fS {\mathfrak{S}}

\def\cC {\mathcal{C}}

\def\cK {\mathcal{K}}

\def\cN {\mathcal{N}}

\def\cR {\mathcal{R}}
\def\cS {\mathcal{S}}
\def\cT {\mathcal{T}}

\def\cZ {\mathcal{Z}}

\def\b {{\beta}}

\def\Ga {{\Gamma}}
\def\de {{\delta}}
\def\eps {{\epsilon}}

\def\l {{\lambda}}
\def\L {{\Lambda}}
\def\si {{\sigma}}
\def\Si {{\Sigma}}

\def\Om {{\Omega}}

\def\d {{\partial}}
\def\grad {{\nabla}}

\def\rstr {{\big |}}
\def\indc {{\bf 1}}

\def\la {\langle}
\def\ra {\rangle}

\def \lA {\big\langle \! \! \big\langle}
\def \rA {\big\rangle \! \! \big\rangle}

\newcommand{\Supp}{\operatorname{supp}}

\newcommand{\Dist}{\operatorname{dist}}

\newcommand{\ba}{\begin{aligned}}
\newcommand{\ea}{\end{aligned}}

\newcommand{\be}{\begin{equation}}
\newcommand{\ee}{\end{equation}}

\newcommand{\lb}{\label}

\newtheorem{Thm}{Theorem}[section]

\newtheorem{Lem}[Thm]{Lemma}

\begin{document}

\title[Boltzmann-Grad Limit for the Lorentz Gas]{The Boltzmann-Grad Limit for the Lorentz Gas with a Poisson Distribution of Obstacles}

\author[F. Golse]{Fran\c cois Golse}
\address[F.G.]{CMLS, \'Ecole polytechnique, 91128 Palaiseau Cedex, France}
\email{francois.golse@polytechnique.edu}

\begin{abstract}
In this note, we propose a slightly different proof of Gallavotti's theorem [``Statistical Mechanics: A Short Treatise'', Springer, 1999, pp. 48--55] on the derivation of the linear Boltzmann equation for the Lorentz gas with a Poisson distribution
of obstacles in the Boltzmann-Grad limit.
\end{abstract}

%\date{\today; PRELIMINARY DRAFT, NOT TO BE DIFFUSED}

\maketitle

\bigskip
\bigskip
\rightline{\it In memory of Prof. Robert T. Glassey (1946-2020)}

%%%%%%%%%%%%%%%%%%%%%%%%%%%%%%%%%%%%%%%%%%%%%%%%%%%%%%%%%%%%%%%%%%%%%%%%%%%%%%%%%%%%%%%%%%%%%%%%%%%%%%%%%%%%%%%%%%%%%%%%%%

\section{Introduction}

%%%%%%%%%%%%%%%%%%%%%%%%%%%%%%%%%%%%%%%%%%%%%%%%%%%%%%%%%%%%%%%%%%%%%%%%%%%%%%%%%%%%%%%%%%%%%%%%%%%%%%%%%%%%%%%%%%%%%%%%%%

In 1905, Lorentz proposed to describe the motion of electrons in a metal by means of a linear Boltzmann equation (equation \eqref{LinBoltzF} below) satisfied by the electron distribution function. The argument in \cite{Lorentz} leading to the 
Lorentz kinetic model closely follows the classical reasoning leading to the Boltzmann equation in the kinetic theory of gases. The essential difference is that the Lorentz collision integral (i.e. the right hand side of \eqref{LinBoltzF}) is linear, 
at variance with the Boltzmann collision integral, which is quadratic.

However, the discussion in \cite{Lorentz} is by no means a rigorous derivation of the Lorentz kinetic model. The first rigorous derivation of the Lorentz equation is due to Gallavotti \cite{Ga69,Ga72} (with \cite{Ga72} remaining unpublished, until
it was included in the book \cite{Gallavotti}). In \cite{Ga72} or \cite{Gallavotti}, the electrons are specularly reflected at the surface of spherical obstacles whose centers are distributed at random and define a Poisson point process in $\bR^3$.
Assuming that the size $\eps>0$ of each obstacle is small, while the number $N$ of obstacles per unit volume is large so that $N\eps^2$ tends to a positive number, which is the \textit{Boltzmann-Grad scaling}, Gallavotti proved in \cite{Ga72} 
or \cite{Gallavotti} that the expected electron distribution function converges to the solution of the Lorentz kinetic equation. Gallavotti's simple, but outstanding result was later generalized by Spohn \cite{Spohn} to more general distributions of 
obstacles. On the other hand, the almost sure (i.e. for almost every denumerable set of obstacle centers)  convergence of the electron distribution function was proved in \cite{Boldri}.

Gallavotti's proof proceeds as follows: for each denumerable set of obstacle centers, the electron distribution function is computed by the method of characteristics, then averaged in the set of obstacle centers (see formula (4.2) in \cite{Ga72}
or formula (1.A2.12) in \cite{Gallavotti}). Then one partitions the set of obstacles centers into configurations where each obstacle is hit only once by the electron, and configurations where at least one obstacle is hit at least twice by the electron 
(see formula (5.6) in \cite{Ga72}, or formula (1.A2.19) in \cite{Gallavotti}). Then one passes to the limit as $\eps$ (the obstacle radius) tends to zero in the expressions so obtained and proves that the contribution of configurations where each 
obstacle is hit only once converges to some expression in the form of a series (formula (5.11) in \cite{Ga72}, or formula (1.A2.23) in \cite{Gallavotti}. One easily recognizes in this series the explicit formula for the solution of the Lorentz kinetic 
equation \eqref{LinBoltzF} in terms of the initial distribution function obtained by iterating on the Duhamel formula (treating the Lorentz collision integral as a perturbation of the free transport equation). The contribution of recollisions is proved
negligible by the conservation of total mass, or particle number satisfied by the Lorentz kinetic equation \eqref{LinBoltzF}.

One remarkable feature of Gallavotti's argument is that the Lorentz equation \eqref{LinBoltzF} is deduced from the explicit formula giving its solution. This is somewhat disconcerting for a mathematician: indeed, solving differential equations 
by means of explicit formulas is quite exceptional. Besides, even if explicit solutions are available, it is usually simpler to deduce the qualitative features of these solutions from the equation they satisfy than from some complicated explicit
formula. 

This situation is far from exceptional in mathematical physics, and there are other fundamental equations which have been derived from a more or less explicit form of their solutions. Perhaps the most famous example is the Boltzmann
equation of the kinetic theory of gases itself. Its first rigorous derivation, due to Lanford \cite{Lanford} (see also \cite{CIP,GallaLSR} for more detailed presentations and generalizations of Lanford's result) follows exactly the same pattern
as Gallavotti's (simpler) derivation of the Lorentz kinetic equation. One shortcoming of this approach is that the time interval on which the Boltzmann equation is derived from Newton's second law written for each gas molecule is extremely
small (a fraction of the average time between two successive collisions involving the same gas molecule).

\smallskip
For this reason, we propose in this note a different approach to the problem of deriving the Lorentz kinetic equation. We shall \textit{never} manipulate any \textit{explicit} formula for the solution of the Lorentz kinetic equation \eqref{LinBoltzF} 
--- or for its Green function. We shall instead obtain the Green function for \eqref{LinBoltzF} as the unique limit point of a carefully designed (see section \ref{S-Green} below) family of Radon measures in the weak-* topology, by a functional 
analytic argument. This Green function is proved to satisfy an integral equation equivalent to the adjoint equation of \eqref{LinBoltzF} (in other words, it is a ``mild'' solution of the adjoint equation of \eqref{LinBoltzF}). 

Perhaps the most important ingredient in this derivation is an analogue of the property referred to as ``one-sided molecular chaos'' in chapter I, section 11 of \cite{Grad58} --- see also Sone's lucid presentation in Appendix A, section A1 of 
\cite{Sone2007} (especially the discussion between equation (A.5) on p. 485 and the Lemma on p. 492, and footnotes 12, 13, 14, of key importance for a good understanding of the foundations of the kinetic theory of gases). This analogue
of Grad's ``one-sided molecular chaos'' idea is presented and explained in detail in section \ref{S-J2} below.

In the end, we obtain exactly the same result as Gallavotti's in \cite{Ga72,Gallavotti}. As explained above, our purpose was certainly not to improve, or generalize Gallavotti's original derivation of the Lorentz equation, but to show that such
a derivation could be carried out by obtaining \textit{directly the integral equation} form of the Lorentz kinetic model \eqref{LinBoltzF}, without \textit{ever using the explicit solution} of \eqref{LinBoltzF}. Hopefully, this approach can be adapted
to derive other kinetic models.

\smallskip
The outline of this paper is as follows: section 2 presents the basic setting of the Lorentz gas, while section 3 reviews some material on the Poisson point process, and section 4 describes the billiard flow, i.e. the Lorentz gas dynamics before
passing to the vanishing $\eps$ limit. The material in sections 2-4 is standard in the theory of the Lorentz gas. Section \ref{S-Green} introduces the ``filtered'' Green function, and its decomposition, \eqref{DefJ1J2}, ultimately leading to the 
integral equation form of the Lorentz kinetic equation in the vanishing $\eps$ limit. The two terms in this decomposition are studied in sections 6 and \ref{S-J2} respectively. The adjoint equation satisfied by the Green function of \eqref{LinBoltzF}
is obtained in section 8, while section 9 contains the end of the proof of Gallavotti's theorem (Theorem \ref{T-Gallavotti} in section 3).

\medskip
Bob Glassey has been a leader in the field of kinetic equations. His work with W. Strauss on the Vlasov-Maxwell system \cite{GlaStr} is an important milestone on the most intriguing open problem in noncollisional kinetic theory, and has had 
a rich posterity. His monograph\cite{GlasseyBook} on the Cauchy problem in kinetic theory is a model of clarity, mathematical precision and elegance. The modest remarks on the foundations of kinetic theory in the present paper are dedicated 
to Bob Glassey's memory in recognition of his influence on the field of kinetic models, both as a researcher and as a mentor.

%%%%%%%%%%%%%%%%%%%%%%%%%%%%%%%%%%%%%%%%%%%%%%%%%%%%%%%%%%%%%%%%%%%%%%%%%%%%%%%%%%%%%%%%%%%%%%%%%%%%%%%%%%%%%%%%%%%%%%%%%%

\section{The Lorentz Gas}

%%%%%%%%%%%%%%%%%%%%%%%%%%%%%%%%%%%%%%%%%%%%%%%%%%%%%%%%%%%%%%%%%%%%%%%%%%%%%%%%%%%%%%%%%%%%%%%%%%%%%%%%%%%%%%%%%%%%%%%%%%

Let $d\ge 2$ be an integer,  and let $C$ be a denumerable subset of $\bR^d$. For $\eps>0$, consider the \textit{billiard table}
\be\lb{BilTable}
\Om_\eps(C):=\{x\in\bR^d\hbox{ s.t. }\Dist(x,C)>\eps\}\,,\quad \cZ_\eps(C):=\Om_\eps(C)\times\bS^{d-1}\,,
\ee
where $\bS^{n-1}$ designates the unit sphere $\{(x_1,\ldots,x_n)\in\bR^n\text{ s.t. }x_1^2+\ldots+x_n^2=1\}$.

If $B\subset\bR^d$, the function 
$$
\bR^d\ni x\mapsto\Dist(x,B)=\inf_{y\in B}\Dist(x,y)\in[0,+\infty)
$$
is easily seen\footnote{Indeed, for each  $y\in B$, one has
$$
\Dist(x_2,B)\le\Dist(x_2,y)\le\Dist(x_2,x_1)+\Dist(x_1,y)\,,
$$
and minimizing the right hand side in $y\in B$ implies that
$$
\Dist(x_2,B)\le\Dist(x_2,x_1)+\Dist(x_1,B)\,.
$$
Exchanging the roles of $x_1,x_2$ leads immediately to the $1$-Lipschitz continuity condition.} to be $1$-Lipschitz continuous, so that $\Om_\eps(C)$ is an open set of $\bR^d$ for each $C\subset\bR^d$ and each $\eps>0$.

If $B(0,R)\cap C$ is finite for all $R>0$, then $\d\Om_\eps(C)\cap B(0,R)$ is a piecewise $C^1$ submanifold of $\bR^d$ for all $R>0$. Indeed, 
$$
\d\Om_\eps(C)\cap B(0,R)=\d(\hbox{finite union of balls of radius $\eps$})\cap B(0,R)\,.
$$
More precisely, an inward unit normal is defined at each point $x\in\d\Om_\eps(C)$ such that 
$$
\#\{c\in C\text{ s.t. }x\in\d B(c,\eps)\}=1\,,
$$
i.e. except for finitely many $x\in\d\Om_\eps(C)$. If this is the case, we denote
\be\lb{Def1c(x)}
\{c_\eps(x)\}:=\{c\in C\text{ s.t. }x\in\d B(c,\eps)\}\,.
\ee

\medskip
The transport equation with specular reflection on $\d\Om_\eps(C)$ for all $C$ denumerable such that $B(0,R)\cap C$ is finite for each $R>0$, is the initial boundary value problem with unknown $f_\eps\equiv f_\eps(t,x,v;C)$ written below
\be\lb{TranspEq1}
\left\{
\ba
{}&\d_tf_\eps(t,x,v;C)+v\cdot\grad_xf_\eps(t,x,v;C)=0\,,&&\qquad x\in\Om_\eps(C)\,,
\\
&f_\eps(t,x,v;C)=f_\eps\left(t,x,v-2\left(v\cdot\tfrac{x-c_\eps(x)}\eps\right)\tfrac{x-c_\eps(x)}\eps;C\right)\,,&&\qquad \#(C\cap\d B(x,\eps))=1\,,
\\
&f_\eps(0,x,v;C)=f^{in}(x,v)\,,&&\qquad x\in\Om_\eps(C)\,.
\ea
\right.
\ee
It is natural to require that $B(0,R)\cap C$ is finite, so that $\d\Om_\eps(C)\cap B(0,R)$ is a piecewise $C^1$ submanifold of $\bR^d$ for all $R>0$. This is indeed the setting in which the initial boundary value problem above for the transport 
equation is well-posed in $L^p(\Om_\eps(C)\times\bS^{d-1})$ (see \cite{BardosThese} and chapter 1 of \cite{Agoshkov}). Later, we shall consider an even simpler situation. 

This problem satisfies the maximum principle: for all $t\ge 0$
\be\lb{MaxPrinc}
0\le f^{in}\le M\quad\hbox{ on }\cZ_\eps(C)\implies 0\le f_\eps(t,x,v;C)\le M\quad\hbox{ for a.e. }(x,v)\in\cZ_\eps(C)\,.
\ee

%%%%%%%%%%%%%%%%%%%%%%%%%%%%%%%%%%%%%%%%%%%%%%%%%%%%%%%%%%%%%%%%%%%%%%%%%%%%%%%%%%%%%%%%%%%%%%%%%%%%%%%%%%%%%%%%%%%%%%%%%%

\section{The Poisson Point Process}

%%%%%%%%%%%%%%%%%%%%%%%%%%%%%%%%%%%%%%%%%%%%%%%%%%%%%%%%%%%%%%%%%%%%%%%%%%%%%%%%%%%%%%%%%%%%%%%%%%%%%%%%%%%%%%%%%%%%%%%%%%

As recalled in the introduction, Gallavotti proves the Boltzmann-Grad limit for a Lorentz gas with spherical scatterers (obstacles) whose centers are distributed according to a Poisson point process in $\bR^3$. This section recalls some basic facts 
about the notion of \textit{Poisson point process}. Standard references for this topic are \cite{Kingman, Daley}; see also the more recent textbook \cite{Last}.

\smallskip
A \textit{Poisson point process} in $\bR^d$ with intensity $\l$ is a random denumerable set $C\subset\bR^d$ such that, for each Borel set $A\subset\bR^d$, the random variable\footnote{For each finite set $S$, we denote by $\# S$ the number of 
elements of $S$, and for each Borel, or Lebesgue measurable subset $A$ of $\bR^d$, we denote by $|A|$ its Lebesgue measure.} 
$$
\#(C\cap A)\text{ has Poisson distribution of parameter }\l |A|,
$$
and, for each $k$-tuple of pairwise disjoints Borel sets $A_1,\ldots,A_k\subset\bR^d$, the integer-valued random variables
$$
\#(C\cap A_1),\,\#(C\cap A_2),\ldots,\,\#(C\cap A_k)\text{ are independent.}
$$
Thus
\be\lb{Poisson1}
\bP(\#(C\cap A_j)=n_j\hbox{ for }j=1,\ldots,k)=\prod_{j=1}^ke^{-\l|A_j|}\frac{\l^{n_j}|A|^{n_j}}{n_j!}\,.
\ee
See section 2.1 in \cite{Kingman}, section 2.4 in \cite{Daley} or Definition 3.1 in \cite{Last}.

\smallskip
Gallavotti starts from a slightly different, but obviously related formula (formula (2.1) in \cite{Ga72}, or formula (1A2.1) in \cite{Gallavotti}). This formula gives the probability of finding $N$ points $c_1,\ldots,c_N$ of $C$ in a Borel subset $B$ of 
$\bR^d$, with $c_1,\ldots,c_N$ in infinitesimal volumes $dc_1,\ldots,dc_N$ centered at $c_1,\ldots,c_N$, which is
\be\lb{Poisson2}
\Pi_{N,B}(c_1,\ldots,c_N)dc_1\ldots dc_N=\frac{\l^Ne^{-\l|B|}}{N!}dc_1\ldots dc_N\,.
\ee
This is called the \textit{Janossy measure of order $N$} of the Poisson point process restricted to $B$: see Definition 4.6, Example 4.8 and formula (4.21) in \cite{Last}, or Definition 5.4 IV in \cite{Daley} (notice the difference in normalization
in these two references).

(Here is a quick informal argument explaining \eqref{Poisson2}: pick $A_1,\ldots,A_N$ pairwise disjoints Borel subsets of $B$, and observe that the probability of finding $N$ particles $c_1,\ldots,c_N$ in $B$, with $c_j\in A_j$ for $j=1,\ldots,N$, 
is exactly the probability of finding one element of $C$ in each $A_j$ for $j=1,\ldots,N$ and no element in $B\setminus (A_1\cup\ldots\cup A_N)$, divided by $N!$ to account for $c_j\in A_{\si(j)}$ with $j=1,\ldots,N$ and $\si\in\fS_N$. According
to formula \eqref{Poisson1}
$$
\ba
\bP(\#(C\cap A_j)=1\hbox{ for }j=1,\ldots,N\text{ and }\#(C\cap B\setminus (A_1\cup\ldots\cup A_N))=\varnothing)
\\
=e^{-\l(|B|-|A_1|-\ldots-|A_N|)}\prod_{j=1}^N\l|A_j|e^{-\l|A_j|}=e^{-\l|B|}\prod_{j=1}^N\l|A_j|
\ea
$$
and this implies formula \eqref{Poisson2}.)

\smallskip
Since\footnote{By the Taylor formula, for each $x\ge 0$ and each $n\ge 1$, 
$$
1-e^{-x}\sum_{k=0}^{n-1}\frac{x^k}{k!}=\int_0^x\frac{(x-t)^{n-1}}{(n-1)!}e^{t-x}dt\le\int_0^x\frac{(x-t)^{n-1}}{(n-1)!}dt=\frac{x^n}{n!}\,.
$$}
\be\lb{PCcapB>n}
\bP(\#(C\cap B(0,R))\ge n)=e^{-\l|B(0,R)|}\sum_{k\ge n}\frac{\l^k|B(0,R)|^k}{k!}\le\frac{\l^n|B(0,R)|^n}{n!}\to 0
\ee
as $n\to\infty$ for each $R>0$, then
$$
\bP(\#(C\cap B(0,N))=\infty)=0\quad\hbox{ for each integer }N>0\,.
$$
Hence
\be\lb{DefcN}
\bP(\cN)=0\,,\quad\hbox{ where }\cN:=\bigcup_{N\ge 1}\{C\subset\bR^d\hbox{ denumerable s.t. }\#(C\cap B(0,N))=\infty\}\,.
\ee
In particular, for all $C\notin\cN$, i.e. $\bP$-a.s., there are finitely many points of $C$ in the ball $B(0,R)$, so that, as explained above, $\d\Om_\eps(C)\cap B(0,R)$ is a piecewise $C^1$ submanifold of $\bR^d$ for all $R>0$.

\medskip
We shall consider below an even simpler situation. For each $\eps>0$, set 
\be\lb{DefcCepsR}
\cC(\eps,R):=\{C\subset\bR^d\text{ denumerable s.t. }c\not= c'\in C\cap B(0,R)\implies|c-c'|>3\eps\}\,.
\ee
Assume that $f^{in}=0$ a.e. on $B(0,R)^c\times\bS^{d-1}$, and that $C\in\cC(\eps,R+T)$. Then, one has $f_\eps(t,x,v)=g_\eps(t,x,v)$ for a.e. $(t,x,v)\in[0,T]\times B(0,R+T)\times\bS^{d-1}$ and for each $T>0$, where $g_\eps$ is the solution 
of the Cauchy problem
\be\lb{TranspEq2}
\left\{
\ba
{}&\d_tg_\eps(t,x,v;C)+v\cdot\grad_xg_\eps(t,x,v;C)=0\,,&&\quad x\in\Om_\eps(C)\cap B(0,R+T)\,,
\\
&g_\eps(t,x,v;C)=g_\eps\left(t,x,v-2\left(v\cdot\tfrac{x-c_\eps(x)}\eps\right)\tfrac{x-c_\eps(x)}\eps;C\right)\,,&&\quad C\cap\d B(x,\eps)=\{c_\eps(x)\}\,,
\\
&g_\eps(t,x,v;C)=0\,,&&\quad |x|=R+T\,,\,\,v\cdot x<0\,,
\\
&g_\eps(0,x,v;C)=f^{in}(x,v)\,,&&\quad x\in\Om_\eps(C)\,.
\ea
\right.
\ee

\smallskip
Notice that, if $C\in\cC(\eps,R+T)$, for each $x\in(\d\Om_\eps(C))\cap B(0,R+T)$ the condition 
\be\lb{Defc(x)}
C\cap\d B(x,\eps)=\{c_\eps(x)\}
\ee
defines $c_\eps(x)$ to be the unique point in $C$ such that $x$ lies on the surface of the \textit{unique} spherical obstacle centered at $c_\eps(x)$. The uniqueness of this obstacle follows from our assumption that $C\in\cC(\eps,R+T)$.

\smallskip
Observe that the spatial domain in this initial boundary value problem is the open set $\Om_\eps(C)\cap B(0,R+T)$. Now $(\d\Om_\eps(C))\cap B(0,R+T)$ is a disjoint union of spheres in $\bR^d$, and therefore a $C^1$ (even a $C^\infty$) 
submanifold of $\bR^d$ --- instead of piecewise smooth submanifold as in the case where $C\notin\cN$. In particular, a unit outward normal vector is uniquely defined at each point of $(\d\Om_\eps(C))\cap B(0,R+T)$ --- and not only at all 
but finitely many points of $(\d\Om_\eps(C))\cap B(0,R+T)$.

\medskip
Define the \textit{first collision time}: for each $(x,v)\in\cZ_\eps(C)$, set
\be\lb{Deftau1}
\tau^1_\eps((x,v);C):=\inf\{t>0\hbox{ s.t. }\Dist(x+tv,C)=\eps\}\,.
\ee
Thus, for each $t\ge 0$,
$$
\ba
\bP(\tau_\eps^1((x,v):C)\ge t)=\bP(C\cap([x,x+tv[+B(0,\eps))=\varnothing)
\\
=e^{-\l|[x,x+tv[+B(0,\eps)|}=e^{-\l\eps^{d-1}|\bB^{d-1}|t}&\,.
\ea
$$
(The notation $\bB^n$ designates the unit ball $\{(x_1,\ldots,x_n)\in\bR^n\text{ s.t. }x_1^2+\ldots+x_n^2\le 1\}$.)
The distribution of $\tau_\eps^1((x,v):C)$ is therefore given by the formula
$$
-d\bP(\tau_\eps^1((x,v),C)\ge t)=\l\eps^{d-1}|\bB^{d-1}|e^{-\l\eps^{d-1}|\bB^{d-1}|t}dt\,.
$$

\smallskip
This formula is at the origin of the \textit{Boltzmann-Grad scaling}:
\be\lb{BGScal}
\l\eps^{d-1}=1\,.
\ee
Defining
\be\lb{Defsi}
\si:=|\bB^{d-1}|\,,
\ee
we see that the distribution of $\tau_\eps^1((x,v),C)$ satisfies
\be\lb{Disttau1}
-d\bP(\tau_\eps^1((x,v),C)\ge t)\to\si e^{-\si t}dt
\ee
in the limit as $\eps\to 0$. 

With this scaling, the \textit{expected number of obstacles per unit volume} is large and tends to infinity as $\eps\to 0$, since
$$
\frac{\bE(\#(C\cap B(0,R)))}{|B(0,R)|}=\frac1{|B(0,R)|}\sum_{n\ge 0}ne^{-\l|B(0,R)|}\frac{\l^n|B(0,R)|^n}{n!}=\l=\eps^{1-d}\,.
$$
On the other hand, the \textit{expected volume fraction} occupied by the obstacles is small and vanishes as $\eps\to 0$, since
$$
\frac{\bE(\#(C\cap B(0,R)))}{|B(0,R)|}|\bB^d|\eps^d=|\bB^d|\eps\,.
$$

\smallskip
After these preliminaries, we can state Gallavotti's theorem.

\begin{Thm}[Gallavotti]\lb{T-Gallavotti}
Let $f^{in}\in C_c(\bR^d\times\bS^{d-1})$ satisfy $f^{in}\ge 0$, and set
$$
F_\eps(t,x,v):=\bE\left[f_\eps(t,x,v;C)\indc_{(x,v)\in\overline{\cZ_\eps(C)}}\right]\,.
$$
Then $F_\eps\to F\text{ weakly-* in }L^\infty([0,+\infty)\times\bR^d\times\bS^{d-1})$ as $\eps\to 0$, where $F$ is the solution of the linear Boltzmann equation
\be\lb{LinBoltzF}
\left\{\ba
{}&(\d_t+v\cdot\grad_x)F(t,x,v)=\tfrac12\int_{\bS^{d-1}}\left(F(t,x,v-2(v\cdot\nu)\nu)-F(t,x,v)\right)|v\cdot\nu|d\nu\,,
\\
&F(0,x,v)=f^{in}(x,v)\,.
\ea
\right.
\ee
\end{Thm}

%%%%%%%%%%%%%%%%%%%%%%%%%%%%%%%%%%%%%%%%%%%%%%%%%%%%%%%%%%%%%%%%%%%%%%%%%%%%%%%%%%%%%%%%%%%%%%%%%%%%%%%%%%%%%%%%%%%%%%%%%%

\section{The Billiard Flow}

%%%%%%%%%%%%%%%%%%%%%%%%%%%%%%%%%%%%%%%%%%%%%%%%%%%%%%%%%%%%%%%%%%%%%%%%%%%%%%%%%%%%%%%%%%%%%%%%%%%%%%%%%%%%%%%%%%%%%%%%%%

Assume that $\Supp(f^{in})\subset B(0,R)\times\bS^{d-1}$ and assume that $C\in\cC(\eps,R+T)$. Then, for each $t\in[0,T]$ and each $z\in\cZ_\eps(C)\cap(B(0,R)\times\bS^{d-1})$,
\be\lb{FlaChar}
f_\eps(t,Z_\eps(t,z;C);C)=f^{in}(z)\,,
\ee
where $Z_\eps(t,z;C)$ is defined by the following prescription.

\smallskip
Denote the phase-space points as $z:=(x,v)\in\bR^d\times\bS^{d-1}$; the free flow acting on the phase-space is 
\be\lb{FreeFlow}
z\mapsto z+t\cT z=(x+tv,v)\,,\qquad\hbox{ with }\cT:=\left(\begin{matrix} 0\,&I_{\bR^d}\\ 0&0\end{matrix}\right)\,.
\ee
For each $x\in\d\Om_\eps(C)\cap B(0,R+T)$, we define $c_\eps(x)\in C$ as in \eqref{Defc(x)}, and set
\be\lb{DefS}
S_\eps(x,v;C):=\left(x,v-2\left(v\cdot\tfrac{x-c_\eps(x)}\eps\right)\tfrac{x-c_\eps(x)}\eps\right)\,.
\ee

\smallskip
For $C\in\cC(\eps,R+T)$, the \textit{billiard flow} is defined on $\cZ_\eps(C)\cap(B(0,R)\times\bS^{d-1})$ for all $t\in[0,T]$ by the following prescription:
\be\lb{DefZ}
\ba
{}&Z_\eps(0,z;C)=z\,,\quad\qquad\qquad\hbox{ for all }z\in\cZ_\eps(C)\cap(B(0,R)\times\bS^{d-1})\,,
\\
&Z_\eps(s,z;C)\!=\!z\!+\!s\cT z\,,\quad\qquad\hbox{ if }0\!<\!s\!\le\!\tau^1_\eps(z;C)\wedge T\text{ and }z\in\cZ_\eps(C)\cap(B(0,R)\times\bS^{d-1})\,,
\\
&Z_\eps(s,z;C)=Z_\eps(s-\tau^1_\eps(z;C)+0,S_\eps(z+\tau^1_\eps(z;C)\cT z;C);C)\,,\qquad\hbox{ if }\tau^1_\eps(z;C)<\!s\!\le T.
\ea
\ee
Henceforth, we denote
$$
Z_\eps(s,z;C):=(X_\eps,V_\eps)(s,x,v;C)\,.
$$

\smallskip
One easily checks the following facts:

\noindent
(a) for all $0\le s\le t\le T$ and all $z\in\cZ_\eps(C)$, one has 
$$
z\in B(0,R+s)\times\bS^{d-1}\implies Z_\eps(t-s,z;C)\in B(0,R+t)\times\bS^{d-1}\,;
$$
(b) for each $C\in\cC(\eps,R+T)$, 
$$
c,c'\in C\cap B(0,R+T)\text{ and }\d B(c,\eps)\cap\d B(c',\eps)\not=\varnothing\implies c=c'\,.
$$
In particular, $\#C\cap\d B(x,\eps)\le1$ for each $x\in B(0,R+T)$ and each $C\in\cC(\eps,R+T)$;

\noindent
(c)  for each $C\in\cC(\eps,R+T)$ and each $z=(x,v)\in\cZ_\eps(C)\cap(B(0,R+T-t)\times\bS^{d-1})$, one has
\be\lb{DefN}
\#\{s\in[0,t]\text{ s.t. }\Dist(X_\eps(s,x,v;C),C)=\eps\}=:N_\eps(z,t;C)<\infty\,.
\ee

\medskip
For each $C\in\cC(\eps,R+T)$, each $z=(x,v)\in\cZ_\eps(C)\cap(B(0,R+s)\times\bS^{d-1})$ and each $0\le s\le t\le T$, denote  
$$
0<\tau_\eps^1(z;C)<\tau_\eps^2(z;C)<\ldots<\tau_\eps^{N_\eps(z,t-s;C)}(z;C)\le t-s
$$
the finite set of collision times defined recursively by the formula
\be\lb{Deftauj}
\tau_\eps^{j+1}(z;C)=\tau_\eps^1(Z_\eps(\tau_\eps^j(z;C)+0,z;C);C)\,,\qquad j=1,\ldots,N_\eps(z,t-s;C)-1\,.
\ee
Likewise, set
\be\lb{Defcj}
c^j_\eps(z;C):=c_\eps(X_\eps(\tau_\eps^j(z;C),z;C))\,,\qquad 1\le j\le N_\eps(z,t-s;C)
\ee
to be the finite sequence of centers of colliding balls in $[0,t-s]$.

Observe that, since $C\in\cC(\eps,R+T)$,
$$
\tau_\eps^{j+1}(z;C)-\tau_\eps^j(z;C)>\eps\qquad\text{ for all }j=1,\ldots,N_\eps(z,t-s;C)-1\,.
$$
In particular, this explains why the set of collision times in $[0,T-s]$ is finite for all $z\in\cZ_\eps(C)\cap(B(0,R+s)\times\bS^{d-1})$ and all $s\in[0,T]$ if $C\in\cC(\eps,R+T)$.

%%%%%%%%%%%%%%%%%%%%%%%%%%%%%%%%%%%%%%%%%%%%%%%%%%%%%%%%%%%%%%%%%%%%%%%%%%%%%%%%%%%%%%%%%%%%%%%%%%%%%%%%%%%%%%%%%%%%%%%%%%

\section{The Filtered Green Function}\lb{S-Green}

%%%%%%%%%%%%%%%%%%%%%%%%%%%%%%%%%%%%%%%%%%%%%%%%%%%%%%%%%%%%%%%%%%%%%%%%%%%%%%%%%%%%%%%%%%%%%%%%%%%%%%%%%%%%%%%%%%%%%%%%%%

For each $t\in[0,T]$ and $z\in\overline{B(0,R+T-t)}\times\bS^{d-1}$, define the measure-valued function $G^{R,T}_\eps(t,z,d\zeta)$ on $\overline{B(0,R+T)}\times\bS^{d-1}$ by the formula
\be\lb{DefGeps}
\ba
\int_{\overline{B(0,R+T)}\times\bS^{d-1}}\phi(\zeta)G^{R,T}_\eps(t,z,d\zeta)
\\
:=\int_{\cC(\eps,R+T)}\phi(Z_\eps(t,z;C))\indc_{z\in\overline{\cZ_\eps(C)}}\prod_{1\le j<k\le N_\eps(z,t;C)}\indc_{c^j_\eps(z,C)\not=c^k_\eps(z,C)}\bP(dC)
\ea
\ee
for all $\phi\in C(\overline{B(0,R+T)}\times\bS^{d-1})$. In other words, setting
\be\lb{Defmueps}
\mu_\eps^{R,T}[t,z;C]:=\de_{Z_\eps(t,z;C)}\indc_{z\in\overline{\cZ_\eps(C)}}\indc_{C\in\cC(\eps,R+T)}\L_\eps(1,N_\eps(z,t;C);z;C)\,,
\ee
where
\be\lb{DefLambda}
\L_\eps(m,n;z;C):=\prod_{m\le j<k\le n}\indc_{c^j_\eps(z;C)\not=c^k_\eps(z;C)}\,,
\ee
one has
$$
G^{R,T}_\eps(t,z,\cdot)=\bE\left(\mu_\eps^{R,T}[t,z;C]\right)\,.
$$

\smallskip
One can think of $\de_{Z_\eps(t,z;C)}\indc_{z\in\overline{\cZ_\eps(C)}}$, or of $\de_{Z_\eps(t,z;C)}\indc_{z\in\overline{\cZ_\eps(C)}}\indc_{C\in\cC(\eps,R+T)}$ as the Green functions for \eqref{TranspEq1} and \eqref{TranspEq2} respectively.
However $\mu_\eps^{R,T}[t,z;C]$ involves the additional factor $\L_\eps(1,N_\eps(z,t;C);z;C)$ which depends not only on $Z_\eps(t,z;C)$, but on $Z_\eps[s,z;C]$ for all $s\in[0,t]$. In particular, $\mu_\eps^{R,T}[t,z;C]$ is not the Green function
of a first order PDE. Nevertheless, it obviously satisfies the inequality
\be\lb{Ineq-mu}
\mu_\eps^{R,T}[t,z;C]\le\de_{Z_\eps(t,z;C)}\indc_{z\in\overline{\cZ_\eps(C)}}\indc_{C\in\cC(\eps,R+T)}\,.
\ee
We shall use this inequality in section \ref{S-Final} to conclude the proof of Theorem \ref{T-Gallavotti}. 

\smallskip
Let us explain the reason for considering $\mu_\eps^{R,T}[t,z;C]$, or its expected value $G^{R,T}_\eps(t,z,\cdot)$. The factor $\L_\eps(1,N_\eps(z,t;C);z;C)$ eliminates all the configurations $C$ of obstacles such that the path $Z_\eps(s,z;C)$
hits the \textit{same} obstacle \textit{at least twice} in the time interval $[0,t]$. In other words, $\L_\eps(1,N_\eps(z,t;C);z;C)$ eliminates \textit{recollisions}, which are known to complicate all derivations of kinetic equations from particle systems. 
Notice however a minor point: to conveniently eliminate recollisions, we must first restrict our attention to obstacle configurations where the minimal distance between two obstacles in $\overline{B(0,R+T)}$ is positive --- in fact, is at least $\eps>0$.

\smallskip
Henceforth, we call $G_\eps^{R,T}$ the ``filtered Green function''. We shall see that the filtered Green function converges on $(0,+\infty)\times(\bR^d\times\bS^{d-1})^2$ in the sense of distributions as $\eps\to 0$, and seek an integral equation 
satisfied by its limit.

\smallskip
Observe that, for each $t\in[0,T]$, one has
\be\lb{DefJ1J2}
\ba
G^{R,T}_\eps(t,z,\cdot)=&\bE\left(\mu_\eps^{R,T}[t,z;C]\indc_{t<\tau_\eps^1(z;C)}\right)+\bE\left(\mu_\eps^{R,T}[t,z;C]\indc_{\tau_\eps^1(z;C)\le t}\right)
\\
=&J_1+J_2\,.
\ea
\ee
The sought integral equation for the filtered Green function in the vanishing $\eps$ limit follows precisely from this decomposition.

\smallskip
In the next two sections, we compute $J_1$ and $J_2$.

%%%%%%%%%%%%%%%%%%%%%%%%%%%%%%%%%%%%%%%%%%%%%%%%%%%%%%%%%%%%%%%%%%%%%%%%%%%%%%%%%%%%%%%%%%%%%%%%%%%%%%%%%%%%%%%%%%%%%%%%%

\section{Computing $J_1$}

%%%%%%%%%%%%%%%%%%%%%%%%%%%%%%%%%%%%%%%%%%%%%%%%%%%%%%%%%%%%%%%%%%%%%%%%%%%%%%%%%%%%%%%%%%%%%%%%%%%%%%%%%%%%%%%%%%%%%%%%%

If $t<\tau_\eps^1(z;C)$, then $Z_\eps(t,z;C)=(I+t\cT)z$ and $N_\eps(z,t;C)=0$, so that 
$$
\L_\eps(1,0;z;C)=\prod_{(j,k)\in\varnothing}\indc_{c^j_\eps(z;C)\not=c^k_\eps(z;C)}=1\,.
$$
Thus
$$
J_1=\de_{z+t\cT z}\bE\left(\indc_{z\in\overline{\cZ_\eps(C)}}\indc_{t<\tau_\eps^1(z;C)}\indc_{C\in\cC(\eps,R+T)}\right)\,.
$$

We have seen in \eqref{Disttau1} that
$$
\bE\left(\indc_{t<\tau_\eps^1(z;C)}\right)=\bP(\tau_\eps^1(z;C)>t)=e^{-\si t}\qquad\text{ as }\eps\to 0\,.
$$
On the other hand
\be\lb{E1tau>t}
\ba
0\le&\bE\left(\indc_{t<\tau_\eps^1(z;C)}\right)-\bE\left(\indc_{z\in\overline{\cZ_\eps(C)}}\indc_{t<\tau_\eps^1(z;C)}\indc_{C\in\cC(\eps,R+T)}\right)
\\
\le&\bE\left(1-\indc_{z\in\overline{\cZ_\eps(C)}}\right)+\bE\left(1-\indc_{C\in\cC(\eps,R+T)}\right)\,.
\ea
\ee

First 
$$
\indc_{z\in\overline{\cZ_\eps(C)}}=\indc_{x\in\overline{\Om_\eps(C)}}=\indc_{\overline{B(x,\eps)}\cap C=\varnothing}
$$
so that, arguing as in \eqref{PCcapB>n},
\be\lb{E1-1Z}
\bE\left(1-\indc_{z\in\overline{\cZ_\eps(C)}}\right)\le\l |B(x,\eps)|=\eps^{1-d}|\bB^d|\eps^d=|\bB^d|\eps\to 0\,,
\ee
where $\bB^d$ is the closed unit ball of $\bR^d$.

The following limit is slightly more involved.

\begin{Lem}\lb{L-1-1C}
Assume (without loss of generality) that $R+T>3$. Then
$$
\bE\left(1-\indc_{C\in\cC(\eps,R+T)}\right)\to 0\quad\text{ as }\eps\to 0\,.
$$
\end{Lem}

\begin{proof}
Indeed, according to formula \eqref{Poisson2},
$$
\ba
\bE\left(\indc_{C\in\cC(\eps,R+T)}\right)=&\bE\left(\indc_{C\cap B(0,R+T)=\varnothing}\right)+\bE\left(\indc_{\#(C\cap B(0,R+T))=1}\right)
\\
&+\sum_{N\ge 2}\bE\left(\indc_{C\cap B(0,R+T)=\{c_1,\ldots,c_N\}}\prod_{1\le j<k\le N}\indc_{|c_j-c_k|>3\eps}\right)
\\
=&e^{-\l|B(0,R+T)|}+\l|B(0,R+T)|e^{-\l|B(0,R+T)|}
\\
&+\sum_{N\ge 2}\int_{B(0,R+T)^N}\prod_{1\le j<k\le N}\indc_{|c_j-c_k|>3\eps}\Pi_{N,B}(c_1,\ldots,c_N)dc_1\ldots dc_N
\\
=&e^{-\l|B(0,R+T)|}\left(1+\l|B(0,R+T)|+\sum_{N\ge 2}\frac{\l^N\Xi_N}{N!}\right)
\ea
$$
where
$$
\Xi_N:=\int_{B(0,R+T)^N}\prod_{1\le j<k\le N}\indc_{|c_j-c_k|>3\eps}dc_1\ldots dc_N\,.
$$
Observe that
$$
\Xi_N=\int_{B(0,R+T)^{N-1}}\prod_{1\le j<k\le N-1}\indc_{|c_j-c_k|>3\eps}\left(\int_{B(0,R+T)}\prod_{l=1}^{N-1}\indc_{|c_l-c_N|>3\eps}dc_N\right)dc_1\ldots dc_{N-1}
$$
and that
$$
\prod_{l=1}^{N-1}\indc_{|c_l-c_N|>3\eps}\ge 1-\sum_{l=1}^{N-1}\indc_{|c_l-c_N|\le3\eps}\,,
$$
so that
$$
\int_{B(0,R+T)}\left(1-\sum_{l=1}^{N-1}\indc_{|c_l-c_N|\le3\eps}\right)dc_N=|\bB^d|((R+T)^d-(N-1)(3\eps)^d)\,.
$$
Therefore
$$
\Xi_N\ge \Xi_{N-1}|\bB^d|((R+T)^d-(N-1)(3\eps)^d)\,.
$$
Hence
$$
\ba
1\ge\bE\left(\indc_{C\in\cC(\eps,R+T)}\right)\ge&e^{-\l|B(0,R+T)|}\sum_{N\ge 0}\frac{\l^N|\bB^d|^N(R+T)^{Nd}}{N!}\prod_{j=1}^{N-1}\left(1-j\tfrac{(3\eps)^d}{(R+T)^d}\right)
\\
\ge&e^{-\l|B(0,R+T)|}\sum_{N\ge 0}\frac{\l^N|\bB^d|^N(R+T)^{Nd}}{N!}\left(1-N\frac{(3\eps)^d}{(R+T)^d}\right)^N
\ea
$$
Denoting for simplicity
$$
\eta:=\left(\tfrac{3\eps}{R+T}\right)^d\,,\quad\b:=|\bB^d|(R+T)^d\,,
$$
one has
$$
1\ge\bE\left(\indc_{C\in\cC(\eps,R+T)}\right)\ge e^{-\l\b}\sum_{N=0}^{m-1}\frac{(\l\b(1-m\eta))^N}{N!}
$$
so that \footnote{Use footnote 3 above and the elementary inequality $n!\ge n^ne^{-n}$ for each integer $n\ge 1$. Indeed, 
$$
n!=n^n\prod_{k=0}^{n-1}(1-\tfrac{k}n)\,,\quad\text{ and }\ln\prod_{k=0}^{n-1}(1-\tfrac{k}n)=\sum_{k=0}^{n-1}\ln(1-\tfrac{k}n)\ge n\int_0^1\ln(1-x)dx=-1
$$
since the function $x\mapsto\ln(1-x)$ is nonincreasing on $(0,1)$.}
$$
\ba
1-\bE\left(\indc_{C\in\cC(\eps,R+T)}\right)\le&1-e^{-\l\b}\sum_{N=0}^{m-1}\frac{(\l\b(1-m\eta))^N}{N!}
\\
\le&\frac{(\l\b(1-m\eta))^m}{m!}\le\frac{(\l\b(1-m\eta))^m}{m^me^{-m}}=(\l\b e(\tfrac1m-\eta))^m
\ea
$$
Assume without loss of generality that $R+T>3$ so that $0<\eta<\eps^d$, and let us choose $m:=[\eps^{-d}]$. Then we find that
$$
(\l\b e(\tfrac1m-\eta))^m\le(\b e(\tfrac{\eps}{\eps^dm}))^m\le(\b e(\tfrac{\eps}{1-\eps^d}))^{\eps^{-d}-1}\to 0\text{ as }\eps\to 0\,.
$$
\end{proof}

\smallskip
Hence, using Lemma \ref{L-1-1C} with \eqref{E1tau>t} and \eqref{E1-1Z} shows that
\be\lb{LimJ1}
J_1\to e^{-\si t}\de_{z+t\cT z}\quad\text{ in total variation as }\eps\to 0^+\,,
\ee
according to \eqref{Disttau1}.
 
%%%%%%%%%%%%%%%%%%%%%%%%%%%%%%%%%%%%%%%%%%%%%%%%%%%%%%%%%%%%%%%%%%%%%%%%%%%%%%%%%%%%%%%%%%%%%%%%%%%%%%%%%%%%%%%%%%%%%%%%%

\section{Computing $J_2$}\lb{S-J2}

%%%%%%%%%%%%%%%%%%%%%%%%%%%%%%%%%%%%%%%%%%%%%%%%%%%%%%%%%%%%%%%%%%%%%%%%%%%%%%%%%%%%%%%%%%%%%%%%%%%%%%%%%%%%%%%%%%%%%%%%%

The key idea in the computation of $J_2$ is to use conditioning on the \textit{first} obstacle. Notice that this idea of conditioning on the first obstacle is the key by which the integral equation (2.9), equivalent to the linear Boltzmann equation 
(2.10), is deduced from the transport process in (2.3) in \cite{Papanico}

\subsection{Conditioning on the first obstacle}
%%%%%%%%%%%%%%%%%%%%%%%%%%%%%%%%%%%%%%%%%%%%%%%%%%%%%%%%%%%%%%%%%%%%%%%%%%%%%%%%%%%%%%%%%%%%%%%%%%%%%%%%%%%%%%%%%%%%%%%%%

In other words, we use the identity
\be\lb{Fla1J2}
J_2=\bE\left(\indc_{\tau_\eps^1(z;C)\le t}\bE\left(\mu_\eps^{R,T}[t,z,C]|c_\eps^1(z;C)\right)\right)\,.
\ee
Observe that, for $C\in\cC(\eps,R+T)$, and for each $z\in\overline{\cZ_\eps(C)}\cap(\overline{B(0,R+T-t)}\times\bS^{d-1})$ such that $\tau_\eps^1(z;C)\le t\le T$, one has
$$
\ba
Z_\eps(t,z;C)=Z_\eps(t-\tau_\eps^1(z;C)+0,S_\eps(z+\tau_\eps^1(z;C)\cT z;C);C)\,,
\\
N_\eps(z,t;C)-1=N_\eps(S_\eps(z+\tau_\eps^1(z;C)\cT z;C),t-\tau_\eps^1(z;C);C)\,,%;S_\eps(z+\tau_\eps^1(z;C)\cT z;C)\,,
\ea
$$
while $Z_\eps(t,z;C)\in\overline{\cZ_\eps(C)}$ for all $t\in[0,T]$ if $z\in\overline{\cZ_\eps(C)}\cap(\overline{B(0,R+T-t)}\times\bS^{d-1})$. Hence
\be\lb{mu-t-shift}
\ba
\mu_\eps^{R,T}[t,z;C]
\\
=\mu_\eps^{R,T}[t-\tau_\eps^1(z;C)+0,S_\eps(z+\tau_\eps^1(z;C)\cT z;C);C]\prod_{k=2}^{N_\eps(z,t;C)}\indc_{c^1_\eps(z;C)\not=c^k_\eps(z;C)}&\,.
\ea
\ee
Moreover $\mu_\eps^{R,T}[t\!-\!\tau_\eps^1(z;C)\!+\!0,S_\eps(z\!+\!\tau_\eps^1(z;C)\cT z;C);C]$ depends only on $c_\eps^j(z;C)$ for $j=2,\ldots,N_\eps(z,t;C)$, and none of this points is equal to $c_\eps^1(z;C)$ unless
$$
\prod_{k=2}^{N_\eps(z,t;C)}\!\!\indc_{c^1_\eps(z;C)\not=c^k_\eps(z;C)}=0\,.
$$

\smallskip
This observation is the exact analogue of Grad's notion of one-sided molecular chaos  in chapter I, section 11 of \cite{Grad58}, and in Appendix A, section A1 of \cite{Sone2007}. In other words, $\mu_\eps^{R,T}[t,z;C]$ depends on $c^1_\eps(z;C)$
only in the following ways:

\smallskip
\noindent
(a) through the presence of $\tau_\eps^1(z;C)$ and of $c_\eps^1(x)$ in the transformation $S_\eps$ defined in \eqref{DefS}, which both appear in the argument of $\mu_\eps^{R,T}$ on the right hand side of \eqref{mu-t-shift}, and

\noindent
(b) through the exclusion factor 
$$
\prod_{k=2}^{N_\eps(z,t;C)}\indc_{c^1_\eps(z;C)\not=c^k_\eps(z;C)}\,,
$$
of negligible statistical weight as $\eps\to 0$.

\smallskip
Let us first dispose of the exclusion factor mentioned in (b). For each $(s,y)$ such that $s\in[0,T]$ and $y\in\overline{B(0,R+T-s)}\times\bS^{d-1}$,
$$
\ba
\bE\left(\mu_\eps^{R,T}[t,z,C]|c_\eps^1(z;C)\right)
\\
=\bE\left(\!\mu_\eps^{R,T}[t\!-\!\tau_\eps^1(z;C)\!+\!0,S_\eps(z\!+\!\tau_\eps^1(z;C)\cT z;C);C]\!\!\prod_{k=2}^{N_\eps(z,t;C)}\!\!\indc_{c^1_\eps(z;C)\not=c^k_\eps(z;C)}\Big|c_\eps^1(z;C)\!\right)
\\
=\lim_{\eta\to 0^+}\bE\Bigg(\mu_\eps^{R,T}[t\!-\!\tau_\eps^1(z;C)\!+\!0,S_\eps(z\!+\!\tau_\eps^1(z;C)\cT z;C);C]
\\
\times\prod_{k=1}^{N_\eps(S_\eps(z+\tau_\eps^1(z;C)\cT z;C),t-\tau_\eps^1(z;C);C)}\!\!\indc_{|c^1_\eps(z;C)-c^k_\eps(S_\eps(z\!+\!\tau_\eps^1(z;C)\cT z;C);C)|>\eta}\Bigg|c_\eps^1(z;C)\Bigg)
\\
=\lim_{\eta\to 0^+}\sum_{M\ge 1}\int_{\overline{B(0,R+T)}^{M-1}}\mu_\eps^{R,T}[s\!+\!0,y;C]\!\prod_{k=1}^{N_\eps(y,s;C)}\!\!\indc_{|c^1_\eps(y,C)-c^k_\eps(y;C)|>\eta}
\\
\times\indc_{N_\eps(y,s;C)\le M-1}\Pi_{M-1,B(0,R+T)\setminus B(c^1_\eps(y,C),\eta)}(c_2,\ldots,c_M)dc_2\ldots dc_M\Bigg|_{s=t\!-\!\tau_\eps^1(z;C)\atop y=S_\eps(z\!+\!\tau_\eps^1(z;C)\cT z;C)}&.
\ea
$$

Now, for a.e. $(s,y)$ such that $s\in[0,T]$ and $y\in\overline{B(0,R+T-s)}\times\bS^{d-1}$ and each $c_1\in B(0,R+T)$, one has 
$$
\ba
\lim_{\eta\to 0^+}\int_{\overline{B(0,R+T)}^{M-1}}\mu_\eps^{R,T}[s\!+\!0,y;C]\!\prod_{k=1}^{N_\eps(y,s;C)}\!\!\indc_{|c_1-c^{k}_\eps(y;C)|>\eta}
\\
\times\indc_{N_\eps(y,s;C)\le M-1}\Pi_{M-1,B(0,R+T)\setminus B(c_1,\eta)}(c_2,\ldots,c_M)dc_2\ldots dc_M
\\
=\int_{\overline{B(0,R+T)}^{M-1}}\mu_\eps^{R,T}[s\!+\!0,y;C]\!\prod_{k=1}^{N_\eps(y,s;C)}\!\!\indc_{c_1\not=c^{k}_\eps(y;C)}
\\
\times\indc_{N_\eps(y,s;C)\le M-1}\Pi_{M-1,B(0,R+T)\setminus\{c_1\}}(c_2,\ldots,c_M)dc_2\ldots dc_M
\\
=\int_{\overline{B(0,R+T)}^{M-1}}\mu_\eps^{R,T}[s\!+\!0,y;C]\!\indc_{N_\eps(y,s;C)\le M-1}\Pi_{M-1,B(0,R+T)}(c_2,\ldots,c_M)dc_2\ldots dc_M
\\
=\bE(\mu_\eps^{R,T}[s\!+\!0,y;C])=G_\eps^{R,T}(s+0,y,\cdot)&\,.
\ea
$$
(To prove the three equalities above, we proceed as follows. The monotone convergence theorem implies the first equality. Then, the second and third equalities follow from the fact that
$$
\Pi_{M-1,B(0,R+T)\setminus\{c_1\}}(c_2,\ldots,c_M)=\Pi_{M-1,B(0,R+T)}(c_2,\ldots,c_M)
$$
for all $c_1\in B(0,R+T)$, as can be seen from \eqref{Poisson2} since $\{c_1\}$ has measure $0$, while
$$
\int_{\overline{B(0,R+T-s)}\times\bS^{d-1}}\mu_\eps^{R,T}[s\!+\!0,y;C]dy\le 1
$$
since $0\le\mu_\eps^{R,T}[s\!+\!0,y;C]\le\de(y-Z_\eps(s+0,y;C))$ according to \eqref{Defmueps}, and
$$
\ba
c^{k}_\eps(y;C)\in\{c_2,\ldots,c_M\}\text{ for }k=1,\ldots,N_\eps(y,s;C)
\\
\implies\prod_{k=1}^{M-1}\indc_{c_1\not=c_k}\le\prod_{k=1}^{N_\eps(y,s;C)}\!\!\!\!\!\indc_{c_1\not=c^{k}_\eps(y;C)}\le 1&\,.
\ea
$$
Therefore
$$
\ba
\int_0^T\int_{\overline{B(0,R+T-s)}\times\bS^{d-1}}\int_{\overline{B(0,R+T)}^{M-1}}\mu_\eps^{R,T}[s\!+\!0,y;C]\left(1-\prod_{k=1}^{N_\eps(y,s;C)}\indc_{c_1\not=c^{k}_\eps(y;C)}\right)
\\
\times\indc_{N_\eps(y,s;C)\le M-1}\Pi_{M-1,B(0,R+T)}(c_2,\ldots,c_M)dc_2\ldots dc_Mdyds
\\
\le\int_{\overline{B(0,R+T)}^{M-1}}\left(1-\prod_{m=2}^{M}\indc_{c_1\not=c_m}\right)\Pi_{M-1,B(0,R+T)}(c_2,\ldots,c_M)dc_2\ldots dc_M
\\
\times\int_0^T\int_{\overline{B(0,R+T-s)}\times\bS^{d-1}}\mu_\eps^{R,T}[s\!+\!0,y;C]dyds
\\
\le T\int_{\overline{B(0,R+T)}^{M-1}}\left(1-\prod_{m=2}^{M}\indc_{c_1\not=c_m}\right)\Pi_{M-1,B(0,R+T)}(c_2,\ldots,c_M)dc_2\ldots dc_M&=0\,,
\ea
$$
where the last equality follows from the formula \eqref{Poisson2} since
$$
\prod_{m=2}^{M}\indc_{c_1\not=c_m}=1\quad\text{ for a.e. }(c_2,\ldots,c_M)\in\bR^{d(M-1)}\,.
$$
Hence
$$
\ba
\int_{\overline{B(0,R+T)}^{M-1}}\mu_\eps^{R,T}[s\!+\!0,y;C]\left(1-\prod_{k=1}^{N_\eps(y,s;C)}\indc_{c_1\not=c^{k}_\eps(y;C)}\right)
\\
\times\indc_{N_\eps(y,s;C)\le M-1}\Pi_{M-1,B(0,R+T)}(c_2,\ldots,c_M)dc_2\ldots dc_M&=0
\ea
$$
for a.e. $(s,y)$ such that $s\in[0,T]$ and $y\in\overline{B(0,R+T-s)}\times\bS^{d-1}$, since the integral of this nonnegative quantity in $(s,y)$ is equal to $0$.)

\smallskip
Hence, by monotone convergence in the series above, 
$$
\ba
\bE\left(\mu_\eps^{R,T}[t,z,C]|c_\eps^1(z;C)\right)
\\
=\bE\left(\mu_\eps^{R,T}[t\!-\!\tau_\eps^1(z;C)\!+\!0,S_\eps(z\!+\!\tau_\eps^1(z;C)\cT z;C);C]\!\prod_{k=2}^{N_\eps(z,t;C)}\!\!\indc_{c^1_\eps(z;C)\not=c^k_\eps(z;C)}\Big|c_\eps^1(z;C)\right)
\\
=\bE\Bigg(S_\eps(I\!+\!\tau_\eps^1(\cdot;C)\cT;C)\#\mu_\eps^{R,T}[t\!-\!\tau_\eps^1(z;C)\!+\!0,z;C);C]
\\
\times\prod_{k=1}^{N_\eps(S_\eps(z+\tau_\eps^1(z;C)\cT z;C),t-\tau_\eps^1(z;C);C)}\!\!\indc_{c^1_\eps(z;C)\not=c^k_\eps(S_\eps(z\!+\!\tau_\eps^1(z;C)\cT z;C);C)}\Bigg|c_\eps^1(z;C)\Bigg)
\\
=G_\eps^{R,T}(t\!-\!\tau_\eps^1(z;C)\!+\!0,S_\eps(z\!+\!\tau_\eps^1(z;C)\cT z;C),\cdot)&\,.
\ea
$$
At this point, we have proved that the contribution of exclusion factor
$$
\prod_{k=2}^{N_\eps(z,t;C)}\indc_{c^1_\eps(z;C)\not=c^k_\eps(z;C)}\,,
$$
is indeed of negligible statistical weight as $\eps\to 0$.

\smallskip
Therefore, we conclude from \eqref{Fla1J2} that
\be\lb{Fla2J2}
\ba
J_2=&\bE\left(\indc_{\tau_\eps^1(z;C)\le t}\bE\left(\mu_\eps^{R,T}[t,z,C]|c_\eps^1(z;C)\right)\right)
\\
=&\bE(G_\eps^{R,T}(t\!-\!\tau_\eps^1(z;C)\!+\!0,S_\eps(z\!+\!\tau_\eps^1(z;C)\cT z;C),\cdot))\,.
\ea
\ee

\subsection{Averaging over the first obstacle}
%%%%%%%%%%%%%%%%%%%%%%%%%%%%%%%%%%%%%%%%%%%%%%%%%%%%%%%%%%%%%%%%%%%%%%%%%%%%%%%%%%%%%%%%%%%%%%%%%%%%%%%%%%%%%%%%%%%%%%%%%

To compute this expectation, we first need the distribution of $c_\eps^1(z;C)$. Observe that
$$
c_\eps^1(z;C)=x+\tau_\eps^1(z;C)v+\eps\sqrt{1-|h|^2}v-\eps h\,,\quad\text{ with }h\perp v\text{ and }|h|\le 1\,.
$$
In the literature on kinetic models, the length of the vector $h\in\bB^{d-1}$ is usually referred to as the \textit{impact parameter}. Hence 
$$
\ba
\text{distribution of }c_\eps^1(z;C)=&\text{distribution of }\tau_\eps^1(z;C)\otimes\text{uniform distribution on }\bB^{d-1}
\\
=&\si e^{-\si t}dt\otimes\text{ uniform distribution on }\bB^{d-1}\,.
\ea
$$

Recall that
$$
\ba
S_\eps(x+\tau_\eps^1(z;C)v,v;C)\\
\\
=\left(x+\tau_\eps^1(z;C)v,v-2\left(v\cdot\tfrac{x+\tau_\eps^1(z;C)v-c_\eps(x+\tau_\eps^1(z;C)v)}\eps\right)\tfrac{x+\tau_\eps^1(z;C)v-c_\eps(x+\tau_\eps^1(z;C)v)}\eps\right)
\\
=\left(x+\tau_\eps^1(z;C)v,\left(1-2\sqrt{1-|h|^2}\right)v+2\sqrt{1-|h|^2}h\right)&\,.
\ea
$$
Hence, denoting by $u$ the uniform probability measure on $\bB^d$,
\be\lb{Fla3J2}
J_2=\int_0^t\int_{\bB^{d-1}}\si e^{-\si\tau}G_\eps^{R,T}(t\!-\!\tau\!+\!0,\cS[\tau,h]z,\cdot)u(dh)d\tau\,,
\ee
where
\be\lb{DefcS}
\cS[\tau,h](x,v)=\left(x+\tau v,\left(2|h|^2-1\right)v+2\sqrt{1-|h|^2}h\right)\,.
\ee

Equivalently
$$
\cS[\tau,h](x,v)=\left(x+\tau v,\cR[\nu]v)\right)\,,
$$
with
$$
\cR[\nu]w:=w-2(w\cdot\nu)\nu\quad\text{ and }\quad\nu:=h-\sqrt{1-|h|^2}v\in\bS^{d-1}\,.
$$
In other words,
$$
\nu=\frac{x+\tau_\eps^1(z;C)v-c_\eps^1(z;C)}\eps
$$
i.e. $\nu$ is the unit inward normal vector of $\d\Om_\eps(C)$ at $x+\tau_\eps^1(z;C)v$ (the first collision point). In particular, $\cS[\tau,h]$ is invertible, and 
$$
\cS[\tau,h]^{-1}(y,w)=(y-\tau\cR[\nu]w,\cR[\nu]w)\,.
$$
In terms of the normal field $\nu$ instead of the impact parameter $h$, one has
\be\lb{Fla4J2}
J_2=\tfrac1{|\bB^{d-1}|}\int_0^t\int_{\bS^{d-1}}\si e^{-\si\tau}G_\eps^{R,T}(t\!-\!\tau\!+\!0,x+\tau v,\cR[\nu]v,\cdot)(v\cdot\nu)_-d\nu d\tau\,.
\ee

%%%%%%%%%%%%%%%%%%%%%%%%%%%%%%%%%%%%%%%%%%%%%%%%%%%%%%%%%%%%%%%%%%%%%%%%%%%%%%%%%%%%%%%%%%%%%%%%%%%%%%%%%%%%%%%%%%%%%%%%%

\section{The Integral Equation for the Green Function}

%%%%%%%%%%%%%%%%%%%%%%%%%%%%%%%%%%%%%%%%%%%%%%%%%%%%%%%%%%%%%%%%%%%%%%%%%%%%%%%%%%%%%%%%%%%%%%%%%%%%%%%%%%%%%%%%%%%%%%%%%

Thus, the equality
$$
G^{R,T}_\eps(t,z,\cdot)=J_1+J_2
$$
can be recast as
\be\lb{IntEqGeps}
G^{R,T}_\eps(t,z,\cdot)-\int_0^t\int_{\bB^{d-1}}\si e^{-\si\tau}G_\eps^{R,T}(t\!-\!\tau\!+\!0,\cS[\tau,h]z,\cdot)u(dh)d\tau=J_1\to e^{-\si t}\de_{z+t\cT z}
\ee
in total variation as $\eps\to 0^+$, for $z=(x,v)$ with $|x|+t\le R+T$ and $t\in[0,T]$.

\smallskip
Next we pass to the limit as $\eps\to 0$ in the left hand side of \eqref{IntEqGeps}.

By construction, $G^{R,T}_\eps(t,z,\cdot)\ge 0$  and $\Supp(G^{R,T}_\eps(t-s,z,\cdot))\subset\overline{B(0,R+t)}\times\bS^{d-1}$ for $0\le s<t\le T$ and $z\in\overline{B(0,R+s)}\times\bS^{d-1}$, and
$$
\int_{\bR^d\times\bS^{d-1}}G_\eps^{R,T}(t,z,d\zeta)=\bE\left(\indc_{C\in\cC(\eps,R+T)}\L_\eps(1,N_\eps(z,t;C);z;C)\right)\le 1\,.
$$

Let 
\be\lb{DefcK}
\cK(R,T):=\{(t,x,v)\in[0,T]\times\bR^d\times\bS^{d-1}\text{ s.t. }|x|\le R+t\}\,,
\ee
and let $\Ga^{R,T}\ge 0$ be a weak-* limit point of the family $G^{R,T}_\eps$ in the space of Radon measures on $\cK(R,T)\times(\overline{B(0,R+T)}\times\bS^{d-1})$. 

For each $g\in C(\cK(R,T))$ and each $\phi\in C(\overline{B(0,R+T)}\times\bS^{d-1})$, one has
\be\lb{IntEq1Ga}
\lA\Ga^{R,T},g\otimes\phi\rA=\int_{\cK(R,T)}e^{-\si t}g(t,z)\phi(z+t\cT z)dzdt+\lA\Ga^{R,T},\Si g\otimes\phi\rA\,,
\ee
where
$$
g\otimes\phi(t,z,\zeta)=g(t,z)\phi(\zeta)\,,
$$
and
$$
\Si g(t,z):=\int_t^T\!\!\!\!\int_{|h|\le 1}\!\!\si e^{-\si(s-t)}\tilde g(s,\cS(s\!-\!t,h)^{-1}z)u(dh)dt\,,
$$
and where $\tilde g$ denotes the extension of $g$ by $0$ in $[0,T]\times\bR^d\times\bS^{d-1}\setminus\cK(R,T)$.

Choosing $\phi\equiv 1$ in the identity above shows that, for each $g\in C(\cK(R,T)$, the measure $m$ defined by
$$
\la m,g\ra:=\la\Ga^{R,T}(t,z,\cdot),g\otimes 1\ra
$$
satisfies
$$
\la m,g\ra=\int_{\cK(R,T)}e^{-\si t}g(t,z)dzdt+\la m,\Si g\ra\,.
$$
In other words, $m$ solves the Cauchy problem 
$$
\left\{\ba
{}&(\d_t-v\cdot\grad_x+\si)m(t,x,v)=\tfrac{\si}{2|\bB^{d-1}|}\int_{\bS^{d-1}}m(t,x,v-2(v\cdot\nu)\nu)|v\cdot\nu|d\nu\,,
\\
&m\rstr_{t=0}=1\,,
\ea
\right.
$$
whose only solution is the measure with density (abusively) denoted
$$
m(t,z)=1\,,\qquad\text{ for a.e. }(t,z)\in[0,T]\,.
$$
Thus $(t,z)\mapsto\Ga^{R,T}$ is a measurable function on $\cK(R,T)$ with values in the set of Borel probability measures on $\overline{B(0,R+T)}$ such that
\be\lb{IntEq2Ga}
\ba
\la\Ga^{R,T}(t,z,\cdot),\phi\ra=&e^{-\si t}\phi(z+t\cT z)
\\
&+\int_0^t\si e^{-\si(t-s)}\int_{|h|\le 1}\la\Ga^{R,T}(s,\cS(t-s,h)z,\cdot),\phi\ra u(dh)ds
\ea
\ee
for each $\phi\in C(\overline{B(0,R+T)})$. In other words, $\Ga^{R,T}(t,z,d\zeta)$ is the measure-valued solution of the Cauchy problem
\be\lb{LinBoltzGa}
\left\{\ba
{}&(\d_t-v\cdot\grad_x+\si)\Ga^{R,T}(t,x,v,\cdot)=\tfrac{\si}{2|\bB^{d-1}|}\int_{\bS^{d-1}}\Ga^{R,T}(t,x,v-2(v\cdot\nu)\nu,\cdot))|v\cdot\nu|d\nu\,,
\\
&\Ga^{R,T}(0,z,\cdot)=\de_z\,.
\ea
\right.
\ee
By the uniqueness of the solution of the Cauchy problem for the linear Boltzmann equation, one concludes that $G^{R,T}_\eps\to\Ga^{R,T}$ as $\eps\to 0$ in the space of Radon measures on $\cK(R,T)\times\overline{B(0,R+T)}$ 
for the weak-* topology.

%%%%%%%%%%%%%%%%%%%%%%%%%%%%%%%%%%%%%%%%%%%%%%%%%%%%%%%%%%%%%%%%%%%%%%%%%%%%%%%%%%%%%%%%%%%%%%%%%%%%%%%%%%%%%%%%%%%%%%%%%

\section{Passing to the Limit in $f_\eps$}\lb{S-Final}

%%%%%%%%%%%%%%%%%%%%%%%%%%%%%%%%%%%%%%%%%%%%%%%%%%%%%%%%%%%%%%%%%%%%%%%%%%%%%%%%%%%%%%%%%%%%%%%%%%%%%%%%%%%%%%%%%%%%%%%%%

We have seen that, if $0\le f^{in}\in C(\bR^d\times\bS^{d-1})$ with $\Supp(f^{in})\subset B(0,R)\times\bS^{d-1}$ and $C\in\cC(\eps,R+T)$, then, for each $t\in[0,T]$ and each $z\in\cZ_\eps(C)\cap(B(0,R)\times\bS^{d-1})$, one has
$$
f_\eps(t,Z_\eps(t,z;C);C)=f^{in}(z)\,.
$$
Thus, for each $0\le\phi\in C(\overline{B(0,R+T)}\times\bS^{d-1})$, one has
$$
\ba
\int_{\bR^d\times\bS^{d-1}}\phi(z)\bE\left(f_\eps(t,z,C)\indc_{z\in\overline{\cZ_\eps(C)}}\right)dz
\\
\ge\int_{\bR^d\times\bS^{d-1}}\int_{\cC(\eps,R+T)}f_\eps(t,z;C)\phi(z)\indc_{z\in\overline{\cZ_\eps(C)}}\bP(dC)dz
\\
=\int_{\bR^d\times\bS^{d-1}}\int_{\cC(\eps,R+T)}f_\eps(t,Z_\eps(t,z;C);C)\phi(Z_\eps(t,z;C))\indc_{z\in\overline{\cZ_\eps(C)}}\bP(dC)dz
\\
=\int_{\bR^d\times\bS^{d-1}}f^{in}(z)\int_{\cC(\eps,R+T)}\phi(Z_\eps(t,z;C))\indc_{z\in\overline{\cZ_\eps(C)}}\bP(dC)dz
\\
\ge\int_{\bR^d\times\bS^{d-1}}f^{in}(z)\int_{\cC(\eps,R+T)}\phi(Z_\eps(t,z;C))\indc_{z\in\overline{\cZ_\eps(C)}}\L_\eps(1,N_\eps(z,t;C);z;C)\bP(dC)dz
\\
=\int_{\bR^d\times\bS^{d-1}}f^{in}(z)\int_{\bR^d\times\bS^{d-1}}\phi(\zeta)G^{R,T}_\eps(t,z,d\zeta)dz&\,.
\ea
$$
(This inequality is equivalent to \eqref{Ineq-mu}, already observed when defining the filtered Green function.)

By the maximum principle \eqref{MaxPrinc} for initial boundary value problem for the transport equation on a domain with piecewise $C^1$ boundary, one has
$$
0\le f_\eps(t,z,C)\le\|f^{in}\|_{L^\infty(\bR^d\times\bS^{d-1})}\quad\text{ for a.e. }z\in\bR^d\times\bS^{d-1}\,,
$$
for all $t\ge 0$ if $C\notin\cN$. Thus, for all $t\ge 0$, one has
$$
0\le\bE\left(f_\eps(t,z,C)\indc_{z\in\overline{\cZ_\eps(C)}}\right)\le\|f^{in}\|_{L^\infty(\bR^d\times\bS^{d-1})}\quad\text{ for a.e. }z\in\bR^d\times\bS^{d-1}\,.
$$
By the Banach-Alaoglu theorem, let $F$ be a limit point of $\bE\left(f_\eps(t,z,C)\indc_{z\in\overline{\cZ_\eps(C)}}\right)$ in $L^\infty([0,\infty)\times\bR^d\times\bS^{d-1})$. Passing to the limit in the inequality above shows that
$$
\int_{\bR^d\times\bS^{d-1}}\phi(z)F(t,z)dz\ge\int_{\bR^d\times\bS^{d-1}}f^{in}(z)\int_{\bR^d\times\bS^{d-1}}\phi(\zeta)\Ga^{R,T}(t,z,d\zeta)dz\,,
$$
so that
\be\lb{F>}
F(t,\cdot)\ge\int_{\bR^d\times\bS^{d-1}}f^{in}(y)\Ga^{R,T}(t,y,\cdot)dy\,.
\ee
We shall prove the converse inequality by an elementary squeezing argument based on the conservation of total mass for the Lorentz kinetic equation.

\smallskip
Indeed, one has on the other hand
$$
\ba
\int_{\bR^d\times\bS^{d-1}}f_\eps(t,z;C)\indc_{z\in\overline{Z_\eps(C)}\times\bS^{d-1}}dz=&\int_{\bR^d\times\bS^{d-1}}f^{in}(z)\indc_{z\in\overline{Z_\eps(C)}\times\bS^{d-1}}dz
\\
\le&\int_{\bR^d\times\bS^{d-1}}f^{in}(z)dz\,,
\ea
$$
so that
$$
\int_{\bR^d\times\bS^{d-1}}\bE\left(f_\eps(t,z;C)\indc_{z\in\overline{Z_\eps(C)}\times\bS^{d-1}}\right)dz\le\int_{\bR^d\times\bS^{d-1}}f^{in}(z)dz\,,
$$
and therefore
\be\lb{intF<}
\int_{\bR^d\times\bS^{d-1}}F(t,z)dz\le\int_{\bR^d\times\bS^{d-1}}f^{in}(z)dz\,.
\ee
Since we have seen that
$$
\int_{\overline{B(0,R+T)}\times\bS^{d-1}}\Ga^{R,T}(t,y,d\zeta)=1
$$
for a.e. $(t,y)\in\cK(R,T)$ (which is precisely the conservation of total mass for \eqref{LinBoltzGa}), we conclude from \eqref{F>} and \eqref{intF<} that
$$
F(t,\cdot)=\int_{\bR^d\times\bS^{d-1}}f^{in}(y)\Ga^{R,T}(t,y,\cdot)dy\,.
$$
By compactness of $\bE\left(f_\eps(t,z,C)\indc_{z\in\overline{\cZ_\eps(C)}}\right)$ in $L^\infty([0,\infty)\times\bR^d\times\bS^{d-1})$ weak-* and uniqueness of its limit point as $\eps\to 0$, 
$$
\bE\left(f_\eps(t,z,C)\indc_{z\in\overline{\cZ_\eps(C)}}\right)\to F
$$
in $L^\infty([0,\infty)\times\bR^d\times\bS^{d-1})$ weak-* as $\eps\to 0$. 

\smallskip
Finally, by linearity of the equation in \eqref{LinBoltzGa}, we conclude that $F$ is the solution to the Cauchy problem for the linear Boltzmann equation
\be\lb{LinBoltzF2}
\left\{\ba
{}&(\d_t+v\cdot\grad_x+\si)F(t,x,v)=\tfrac{\si}{2|\bB^{d-1}|}\int_{\bS^{d-1}}F(t,x,v-2(v\cdot\nu)\nu)|v\cdot\nu|d\nu\,,
\\
&F(0,x,v)=f^{in}(x,v)\,,
\ea
\right.
\ee
which concludes the proof of Gallavotti's theorem.

%%%%%%%%%%%%%%%%%%%%%%%%%%%%%%%%%%%%%%%%%%%%%%%%%%%%%%%%%%%%%%%%%%%%%%%%%%%%%%%%%%%%%%%%%%%%%%%%%%%%%%%%%%%%%%%%%%%%%%%%%

\bigskip
\noindent
\textbf{Acknowledgements.}  I have had the opportunity of discussing on several occasions with Prof. M. Pulvirenti the problem of deriving kinetic equations from Newton's second law written for each particle, by working directly on 
the dynamical equations, and not on any explicit formula for their solutions. I am also very grateful to Prof. Y. Sone, who explained to me in detail his presentation of the derivation of the Boltzmann equation in Appendix A1 of his book 
\cite{Sone2007} --- and especially Grad's notion of ``one-sided molecular chaos'' condition.

%%%%%%%%%%%%%%%%%%%%%%%%%%%%%%%%%%%%%%%%%%%%%%%%%%%%%%%%%%%%%%%%%%%%%%%%%%%%%%%%%%%%%%%%%%%%%%%%%%%%%%%%%%%%%%%%%%%%%%%%%

\end{document}